%% file: cylinder.tex
\documentclass[reqno, 12pt]{amsart}

\usepackage[utf8]{inputenc}
\usepackage[T1]{fontenc}
\usepackage{lmodern}
\usepackage{amsmath}
\usepackage{amsthm}
\usepackage{amssymb}
\usepackage{enumerate}
\usepackage{mathrsfs}
\usepackage{mathscinet}
\usepackage{stmaryrd} % crochets nombres entiers
\usepackage{graphicx}
\usepackage[all]{xy}
\usepackage{url}
\usepackage{hyperref}
\usepackage[a4paper, hmargin=3cm, vmargin=3cm]{geometry}
\usepackage{color}

\newcommand\mathens[1]{\mathbb{#1}} %fonte des ensembles classiques
\newcommand{\ud}{\mathrm{d}}

\newcommand{\N}{\mathens{N}}
\newcommand{\Z}{\mathens{Z}}

\newcommand{\R}{\mathens{R}}

\newcommand{\J}{\mathens{J}}

\newcommand{\lochom}{\textup{C}}
\newcommand\interior[1]{\textup{Int}(#1)}

\DeclareMathOperator\ind{ind}
\DeclareMathOperator\mind{\overline{ind}}

\DeclareMathOperator{\dist}{dist}

\newtheorem{thm}{Theorem}[section]
\newtheorem{lem}[thm]{Lemma}
\newtheorem{cor}[thm]{Corollary}

\newtheorem{prop-def}[thm]{Definition-proposition}

\theoremstyle{definition}

\theoremstyle{remark}

\begin{document}

\title
{Homologically visible closed geodesics on complete surfaces}

\author[S. Allais]{Simon Allais}
\author[T. Soethe]{Tobias Soethe}
\address{Simon Allais,
\'Ecole Normale Sup\'erieure de Lyon,
UMPA\newline\indent  46 all\'ee d'Italie,
69364 Lyon Cedex 07, France}
\email{simon.allais@ens-lyon.fr}
\urladdr{http://perso.ens-lyon.fr/simon.allais/}
\address{Tobias Soethe,
    Ruhr-Universität,
    Fakultät für Mathematik
    \newline\indent Universitätsstr. 150,
44780 Bochum, Deutschland}
\email{Tobias.Soethe@ruhr-uni-bochum.de}
\date{May 11, 2020}

\begin{abstract}
    In this article, we give multiple situations
    when having one or two geometrically distinct closed geodesics on a complete
    Riemannian cylinder $M\simeq S^1\times\R$ or a complete Riemannian
    plane $M\simeq\R^2$ leads to having
    infinitely many geometrically distinct closed geodesics.
    In particular, we prove that any complete cylinder with isolated closed
    geodesics
    has zero, one or infinitely many homologically visible closed geodesics;
    this answers a question of Alberto Abbondandolo. 
\end{abstract}
\maketitle

\section{Introduction}

The problem of the existence and multiplicity of closed geodesics plays an important
role in both Riemannian geometry and dynamics.
Going back to Hadamard and Poincaré \cite{Had98,Poi05},
it is still open for a lot of Riemannian manifolds.
Given a complete Riemannian manifold $(M,g)$,
a famous question is whether it possesses a closed geodesic for every
Riemannian metric $g$.
This is always true if $M$ is closed \cite{Bir66, LF51, Fet52}.
We can then ask whether the number of closed geodesics
is infinite or not.
It is known that every closed surface has infinitely many
geometrically distinct closed geodesics \cite{Fra92, Ban93, Hin93}.
However, this question is still open for spheres of higher dimension.
In this article, we are interested in non-compact
complete Riemannian surfaces for which even the existence
of one closed geodesic fails in general: planes and cylinders.
For instance, the Euclidean plane does not possess any closed geodesic.
Nevertheless, under specific geometric conditions, interesting
results can be stated.
In 1980, Bangert proved that any complete Riemannian cylinder, plane or
Möbius band of finite area has infinitely many closed geodesics \cite{Ban80}.
For the plane and the cylinder he proved the same result even under the weaker assumption
of just the existence of a convex neighborhood of infinity. We will discuss this result
in greater depth as it is used extensively in our proofs.
The purpose of this article is to give simple conditions
under which the existence of one or two distinct closed geodesics
implies that a complete Riemannian cylinder or plane contains infinitely many 
geometrically distinct closed geodesics.

Let $S^1:=\R/\Z$ and let
$M\simeq S^1\times\R$ be a complete Riemannian cylinder.
Let $\Lambda M$ be its loop space.
Two loops $\alpha,\beta\in\Lambda M$ are said to be geometrically
distinct if their images are distinct: $\alpha(S^1)\neq\beta(S^1)$.
Throughout the article, by writing that two closed geodesics are distinct
we will always mean that they are geometrically distinct.
A closed geodesic $\gamma\in\Lambda M$ is said to be homologically
visible if the local homology of the critical circle
$S^1\cdot\gamma \subset \Lambda M$ of the energy functional is non-zero
(see Section~\ref{se:preliminaries} for precise definitions).

\begin{thm}\label{thm:main}
    Let $M$ be a complete Riemannian cylinder with isolated closed geodesics
    and assume one of the
    following hypothesis:
    \begin{enumerate}[1.]
        \item\label{it:contractible}
            there exists a contractible closed geodesic,
        \item\label{it:selfintersecting}
            there exists a self-intersecting closed geodesic,
        \item\label{it:intersecting}
            there exist two distinct closed geodesics
            that intersect,
        \item\label{it:mind}
            there exists a closed geodesic of non-zero average index,
        \item\label{it:homological}
            there exist two homologically visible closed geodesics.
    \end{enumerate}
    Then $M$ contains infinitely many
    geometrically distinct homologically visible closed geodesics.
\end{thm}

Notice that according to Bott iteration theory,
a closed geodesic $c$ has a non-zero average index
if and only if some iterate $c^m$ has a non-zero index.
The fact that hypothesis \ref{it:homological} implies that there exists
infinitely many
closed geodesics proves a conjecture of Abbondandolo:

\begin{cor}
    Any complete Riemannian cylinder with isolated closed geodesics has zero, one or
    infinitely many homologically visible
    closed geodesics.
\end{cor}

Similar results can also be obtained when $M\simeq\R^2$ is a complete plane:

\begin{thm}\label{thm:plane}
    Let $M$ be a complete Riemannian plane with isolated closed geodesics
    and assume one of the following hypothesis:
    \begin{enumerate}[1.]
        \item\label{it:plane:selfintersecting}
            there exists a self-intersecting closed geodesic,
        \item\label{it:plane:intersecting}
            there exist two distinct closed geodesics
            that intersect,
        \item\label{it:plane:mind}
            there exists a closed geodesic of non-zero average index,
        \item\label{it:plane:homological}
            there exists a homologically visible closed geodesic.
    \end{enumerate}
    Then $M$ contains infinitely many geometrically distinct
    homologically visible closed geodesics.
\end{thm}

\begin{cor}
    Any complete Riemannian plane with isolated closed geodesics has zero or infinitely
    many geometrically distinct homologically visible closed geodesics.
\end{cor}
  
It is easy to give counter-examples to Theorem~\ref{thm:main}
when none of the assumptions \ref{it:contractible}--\ref{it:homological} hold
by considering embedded cylinders of revolution
\begin{equation}\label{eq:revolution}
    (\theta,z)\mapsto (r(z)\cos\theta,r(z)\sin\theta,z),
\end{equation}
for well-chosen smooth maps $r:\R\to (0,+\infty)$.
A complete cylinder may have no closed geodesic at all:
take $r'>0$.
It can have an arbitrary large finite number $k\in\N$
of homologically invisible closed geodesics:
take $r'(z)>0$ for all $z\in\R\setminus\{z_1,\ldots,z_k\}$
and $r'(z_i)=0$. 
It can also have a unique visible closed geodesic:
take $r'<0$ on $(-\infty,0)$, $r'(0)=0$ and
$r'>0$ on $(0,+\infty)$
(one can as well add to this cylinder an arbitrary large finite number
of homologically invisible closed geodesic
the same way as before).
Remark that in our examples closed geodesics
are without self-intersections and not contractible
as implied by the theorem.
Counter-examples where the theorem fails by lack of completeness
can be found as well by choosing embedded cylinders of revolution
restricting the domain of the embedding (\ref{eq:revolution})
to $(\theta,z)\in S^1\times (a,b)$ for $a,b\in\R$.
We could proceed as follows:
take an even $r:[-1,1]\to(0,+\infty)$ with
$r'>0$ on $[-1,0)$ such that $z=0$ is the only closed geodesic of
the associated compact embedded cylinder.
One can find such an $r$ by slightly modifying a Tannery surface:
a sufficient condition is
that the metric $g$ on the interior of the cylinder can be written as
\begin{equation*}
    g = \left[ \alpha + h(\cos\rho)\right]^2\ud\rho^2 
    + \sin^2 \rho\,\ud\theta^2,
\end{equation*}
for a good choice of coordinates $(\rho,\theta)\in(0,\pi)\times S^1$,
where $\alpha$ is irrational and $h:(-1,1)\to(-\alpha,\alpha)$ is a smooth
odd function (see for instance \cite[Theorem~4.13]{Bes78}).
Then extend $r$ to a smooth map $(-3,1]\to(0,+\infty)$
with $r|_{(-3,-1)} < r(-1)$, $r'<0$ on $(-3,-2)$
and $r'>0$ on $(-2,-1)$.
Then $z=-2$ and $z=0$ are the only closed geodesic
of the cylinder embedded by $r|_{(-3,1)}$ and
are both visible.

In a similar way, we can give examples of complete planes
with nothing but an arbitrary finite number of homologically
invisible closed geodesics by considering surfaces of revolution
(\ref{eq:revolution}) parametrized by $\R/2\pi\Z\times [0,+\infty)$
with $r:[0,+\infty)\to [0,+\infty)$ being increasing and
smooth on $(0,+\infty)$ with $r(0)=0$ and $r'(z)\to +\infty$
when $z\to 0$ in a suitable way (\emph{i.e.} so that the surface
is smooth at the origin).
Then, as above, we get homologically invisible closed geodesics
on the inflexion points of $r$, and nowhere else.

We say that $C_-\subset M$ (resp. $C_+$) is a neighborhood of $-\infty$
(resp. of $+\infty$) if $C_-$ contains $S^1\times (-\infty,a)$
for some $a\in\R$ (resp. $S^1\times (b,+\infty)$ for some $b\in\R$)
for an arbitrarily fixed identification of $M$ with $S^1\times\R$.
In order to prove Theorem~\ref{thm:main}, we will extensively
use the following theorem due to Bangert:

\begin{thm}[{\cite[Theorem~3, Remark~2]{Ban80}}]\label{thm:Bangert}
    Let $M$ be a complete Riemannian cylinder with isolated closed geodesics
    and suppose there
    exists disjoint locally convex open neighborhoods $C_-$ and $C_+$
    of $-\infty$ and $+\infty$ respectively such that the boundaries
    $\partial C_\pm$ are not totally geodesic.
    Then $M$ contains infinitely many homologically visible closed geodesics
    intersecting $M\setminus (C_-\cup C_+)$ and at least one
    without self-intersections.
\end{thm}

Since Bangert did not give the precise proof of that statement,
for the sake of completeness we give a comprehensive proof in the paper.
The proof of Theorem~\ref{thm:plane} is quite similar and
relies extensively on the analogous theorem of Bangert when
$M$ is a plane with isolated closed geodesics:
if there exists an open neighborhood of infinity $C\neq M$
with a boundary $\partial C$ which is not totally geodesic,
$M$ contains infinitely many homologically visible closed geodesics
\cite[Theorem~3]{Ban80}.
These two theorems
were originally used by Bangert to prove that any complete Riemannian
plane of finite area has infinitely many closed geodesics.

In fact, Theorem~\ref{thm:main} extends \emph{verbatim} to the case where
$M$ is a complete reversible Finsler manifold
as we will essentially use variational properties of geodesics
in our proof with no concern for geometric notion
specific to Riemannian manifold.
However, nothing can be said concerning the more general case
of a complete (asymmetrical) Finsler manifold.
The major issue is that, in the asymmetrical case,
a closed subset of $M$ which is bounded by a geodesic
is not locally convex. 
In this direction, we point out that the related question of whether or not
infinitely many closed geodesics exist on every irreversible Finsler cylinder
of finite area is still open \cite[Question~2.3.2]{BM19}.

In order to put these results in perspective, we recall some
known results concerning existence of closed geodesics on complete non-compact
Riemannian manifolds.
In 1978, Thorbergsson proved the existence of closed geodesics on a complete
Riemannian manifold $M$ if it contains a convex compact set which
is not homotopically trivial or if $M$ has a non-negative sectional
curvature outside some compact set \cite{Tho78}.
In the 1990s, Benci and Giannoni proved that
any complete $d$-dimensional Riemannian manifold such that
the limit superior of its sectional curvature at infinity is non-positive
and the homology of its free loop space is non-trivial in some degree larger
than $2d$ possesses a closed geodesic
\cite{BG91, BG92}.
In 2017, Asselle and Mazzucchelli showed the existence of infinitely
many closed geodesics for complete $d$-dimensional Riemannian manifolds which have
no close conjugate points at infinity and a free loop space
with unbounded Betti numbers in degrees larger than $d$ \cite{AM17}.
They also improved the result of Benci and Giannoni by replacing the asymptotic
curvature assumption by an assumption on the conjugated points at infinity
and by improving the bound on the homology of the free loop space.
However, the existence of one closed geodesic in any complete Riemannian manifold
of finite volume is still an open problem
(see for instance the following recent review of the subject
\cite{BM19}).
\subsection*{Organization of the paper}
In Section~\ref{se:preliminaries}, we fix notation and
recall results of the variational theory of geodesics that
we will need.
In Section~\ref{se:Bangert}, we give a comprehensive proof of Theorem~\ref{thm:Bangert}
of Bangert.
In Section~\ref{se:contractible}, we prove Theorem~\ref{thm:main}
when hypothesis \ref{it:contractible}, \ref{it:selfintersecting}
or \ref{it:intersecting} is assumed.
In Section~\ref{se:mind}, we prove Theorem~\ref{thm:main}
when hypothesis \ref{it:mind} is assumed.
In Section~\ref{se:homological}, we prove the last case
of Theorem~\ref{thm:main}.
In Section~\ref{se:plane}, we prove Theorem~\ref{thm:plane}.

\subsection*{Acknowledgments}
This work started while the first author was visiting the Ruhr-Universität Bochum
(Germany) in October 2019. The first author wishes to thank Alberto Abbondandolo and
Stefan Suhr for the invitation, and the Ruhr-Universität Bochum for providing
an excellent working environment.
The first author is also grateful to his advisor Marco Mazzucchelli
who introduced him to the result conjectured by Alberto Abbondandolo.
To the latter the second author expresses his gratitude for his continuing help
and advice. The second author is partially supported by the SFB/TRR 191 “Symplectic Structures in Geometry, Algebra and Dynamics”, funded by the Deutsche Forschungsgemeinschaft.

\section{Preliminaries}\label{se:preliminaries}

In this section, we recall some results of Riemannian geometry
that we will use in our proofs and fix some notation.
For the extension of these notions to the Finsler case,
the reader may for instance look at \cite[Section~2]{CJ13}.

\subsection{The energy functional}

Given a complete Riemannian manifold with boundary $M$, we denote by
$\Lambda M$ the space of $H^1$-maps $S^1\to M$.
In fact, if one wants to avoid analytic questions,
we can always reduce our space to a finite-dimensional
manifold of broken geodesics.
For $\gamma\in\Lambda M$ and $m\in\N^*$,
the iterated loop $\gamma^m\in\Lambda M$ 
is defined by $t\mapsto \gamma(mt)$.
A geodesic is an immersed path $\gamma:\R\to M$ such that
\begin{equation*}
    \nabla_{\dot{\gamma}}\dot{\gamma} = 0,
\end{equation*}
where $\nabla$ denotes the Levi-Civita connection of the metric
and $\dot{\gamma}$ stands for the derivative of $\gamma$.
Therefore, in our convention, geodesics have constant speed.
A closed geodesic is a geodesic $\gamma$ which is periodic:
$\gamma(t+1)=\gamma(t)$ so that $\gamma\in\Lambda M$.
Closed geodesics of $M$ are the critical points with non-zero critical value of
the energy functional $E:\Lambda M\to [0,+\infty)$,
\begin{equation*}
    E(\gamma) := \int_{S^1} g_{\gamma}(\dot{\gamma},\dot{\gamma})\ud t,
    \quad \forall \gamma\in\Lambda M.
\end{equation*}
The energy functional $E$ is $C^2$.
If $M$ is a locally convex compact manifold (possibly with boundary),
$E$ also satisfies the Palais-Smale condition and the $(-\nabla E)$-flow
is defined for all time $t\geq 0$.
We notice that every closed geodesic lies on a critical circle
$S^1\cdot\gamma$, where $S^1$ acts on $\Lambda M$ by $t\cdot\gamma :=
\gamma(t+\cdot)$.
In our study we assume that $E$ has only isolated critical circles
(except for the constant loops which have zero value).
Two closed geodesics $c_1$ and $c_2$ are said to be geometrically distinct if 
they do not have the same image in $M$.

\subsection{Finite-dimensional approximation of the loop space}

Morse's finite-dimen\-sional approximation of the curve space over $M$, as
presented by Bangert in
\cite{Ban80} consist of the following data: an open set $\mathcal{O}\subset M$, an energy bound $\kappa>0$ and a
parameter $j\in\N$ satisfying $\frac{1}{j} <
\frac{\varepsilon^2}{\kappa}$
where $\varepsilon > 0$ is smaller than the injectivity radius on
$\mathcal{O}$. The positivity of $\varepsilon$ will be fulfilled if
for instance $\mathcal{O}$
has compact closure, as will be the case in our considerations. The
finite-dimensional approximation $\Omega=\Omega(\mathcal{O},\kappa,j)$ is
constructed as follows: it is the set of all curves $\gamma\in\Lambda M$ such
that $E(\gamma)<\kappa$, $\gamma(i/j)\in\mathcal{O}$ and such that
$\gamma|_{[i/j,(i+1)/j]}$ is a geodesic of length less than $\varepsilon$
for $0\leq i\leq j-1$.

Let $\Omega$ be a finite-dimensional approximation of $\Lambda M$ and $C\subset
M$ a locally convex set with compact boundary such that $C \subset
\mathcal{O}$. By the local convexity of $C$, there exists an $\varepsilon>0$
such that for two points $p,q \in
C$ with Riemannian distance $d(p,q)<\varepsilon$,
there exists a unique geodesic of length $=d(p,q)$ joining $p$ and $q$, and contained
entirely in $C$. The negative gradient of the restriction of
the energy functional to $\Omega$ is given by
\begin{equation*}
    -\nabla E|_{\Omega}(\gamma) = -2
    \big(\dot{\gamma}_1(1/j)-\dot{\gamma}_2(1/j),\ldots,
    \dot{\gamma}_{j-1}((j-1)/j)-\dot{\gamma}_{j}((j-1)/j)\big)
\end{equation*}
for $\gamma \in \Omega$, where 
$\gamma_i = \gamma|_{[(i-1)/j,i/j]}$ for $1 \le i \le j$ (see
\cite[p.~252]{GM68}). Now from our choice of $j$
and Cauchy-Schwarz inequality, we get
\begin{align*}
    d(\gamma((i-1)/j),\gamma(i/j))^2 \le
    \frac{1}{j}E(\gamma|_{((i-1)/j,i/j)}) \le
    \frac{\varepsilon^2}{\kappa}\kappa=\varepsilon^2
\end{align*}
and consequently by local convexity of $C$, the negative gradient flow of the 
finite\-di\-men\-sio\-nal approximation of the energy functional respects $\Omega$.

\subsection{Index of a closed geodesic}

The index of a closed geodesic $\gamma$ is the Morse index of
$E$:
\begin{equation*}
    \ind(\gamma) := \ind(E,\gamma).
\end{equation*}
It is always finite.
The behavior of this index under iteration
$k\mapsto \ind(\gamma^k)$ was extensively studied by Bott
in \cite{Bot56}.
We simply recall that
\begin{equation}\label{eq:Bott}
    \ind(\gamma^k) \geq k\mind(\gamma) - \dim(M) +1,
    \quad k\in\N,
\end{equation}
where $\mind(\gamma)\geq 0$ is the average index of $\gamma$ defined by
\begin{equation*}
    \mind(\gamma) := \lim_{k\to\infty} \frac{\ind(\gamma^k)}{k}.
\end{equation*}
Let $p\in M$ and $\Omega_p M\subset \Lambda M$ be the set of
loops based at $p$, that is $H^1$-paths $\gamma:[0,1]\to M$ such that
$\gamma(0)=\gamma(1)=p$.
Given a closed geodesic $\gamma\in\Lambda M$,
we denote by $\ind_\Omega(\gamma)\in\N$ the Morse index
\begin{equation*}
    \ind_\Omega (\gamma) := \ind\left(E|_{\Omega_{\gamma(0)}M},\gamma\right).
\end{equation*}
By inclusion, $\ind_\Omega (\gamma)\leq \ind(\gamma)$.
In fact, we have the concavity inequality
\cite[Eq.~(1.5)]{BTZ}:
\begin{equation}\label{eq:conc}
    \ind(\gamma) - \dim(M) +1 \leq
    \ind_\Omega(\gamma)\leq
    \ind(\gamma).
\end{equation}

A Jacobi field of the geodesic path $\gamma$ is a smooth map
$J:\R\to\gamma^*TM$, satisfying
\begin{equation*}
    J(t)\in T_{\gamma(t)}M, \quad \forall t\in\R \quad \text{and} \quad
    \nabla_{\dot{\gamma}}^2 J = R(\dot{\gamma},J)\dot{\gamma},
\end{equation*}
where $R$ denotes the Riemann tensor.
Let $\mu(t)$ be the number of linearly independent Jacobi fields of $\gamma$
such that $J(0)=J(t)=0$; the Morse index theorem states that
\begin{equation}\label{eq:MorseIndexThm}
    \ind_\Omega (\gamma) = \sum_{0<t<1} \mu(t).
\end{equation}

The local homology of a critical circle $S^1\cdot\gamma$ is by definition
\begin{equation*}
    \lochom_*(S^1\cdot\gamma) := H_* (\{ E < E(\gamma) \} \cup S^1\cdot \gamma,
    \{ E < E(\gamma) \}),
\end{equation*}
where the set $\{ E<E(\gamma)\}$ is
$\{\delta\in\Lambda M\ |\ E(\delta)<E(\gamma)\}$,
and $H_*$ denotes the singular homology with integral coefficients.
A closed geodesic is said to be homologically visible if
$\lochom_*(S^1\cdot\gamma)\neq 0$ and is said to be homologically invisible
otherwise.
We will be interested in the properties of the local homology
$\lochom_*(S^1\cdot \gamma)$
only in the case where $\gamma$ is a closed geodesic of
average index $\mind(\gamma)=0$ and whose image $\gamma(S^1)$ lies
in the interior of $M$
($\mind(\gamma)=0$ is equivalent to the fact that $\ind(\gamma^m)$
vanishes for all $m\geq 1$).
Let $\gamma\in\Lambda M$ be such a closed geodesic.
Given $m\in\N$, we denote by $\psi_m : \Lambda M\to\Lambda M$ 
the iteration map $\psi_m(\delta):=\delta^m$.
According to a theorem of Gromoll-Meyer
\cite[Theorem~3]{GM69a},
the local homology of $\lochom_d(S^1\cdot \gamma)$ is zero in degrees 
$d\geq 2\dim M$ and
there exists infinitely many positive integers
$m$ such that the induced map in homology
\begin{equation}\label{eq:GM}
    (\psi_m)_* : \lochom_*(S^1\cdot \gamma) \to \lochom_*(S^1\cdot \gamma^m)
\end{equation}
is an isomorphism.
On the other hand, a theorem of Bangert-Klingenberg
\cite[Corollary~1]{BK83} states that 
there exists $m_0\in\N$ above which for all
$m\geq m_0$, there exists $e_m > m^2 E(\gamma)$ such that the composition
\begin{equation}\label{eq:BK}
    \lochom_*(S^1\cdot \gamma) \xrightarrow{(\psi_m)_*}
    \lochom_*(S^1\cdot \gamma^m) \xrightarrow{\textup{inc}_*}
    H_*\left(\big\{ E < e_m\big\},
    \left\{ E< m^2 E(\gamma) \right\} \right)
\end{equation}
is zero.

\section{Proof of Bangert theorem}\label{se:Bangert}

A closed geodesic $\gamma$ is a mountain pass if, for all neighborhoods
$U\subset\Lambda M$ of $S^1\cdot\gamma$, $U\cap E^{-1}([0,E(\gamma))$ is not connected.
For the proof of Theorem \ref{thm:Bangert}, we need the following statement,
which tells us that isolated closed geodesics cannot remain mountain pass
critical points of the energy functional when sufficiently iterated. A
geometric proof is given by Bangert \cite{Ban80}.

\begin{thm}[{\cite[Theorem~2]{Ban80}}]
    \label{thm:mountainpass}
    Let $\gamma$ be an isolated closed geodesic on $M$, where $\dim M=2$. Then there
    exists an integer $m_\gamma\in\N$ such that the following holds: For all
    integer $m\in\N$ with $m\ge m_\gamma$, there is a neighborhood $U$ of $S^1
    \cdot \gamma$ in $\Lambda M$ such that $U \cap E^{-1}([0,E(\gamma^m)))$ is
    connected.
\end{thm}

According to Gromoll-Meyer \cite{GM68}, given
an isolated closed geodesic $\gamma$, there exists a connected
neighborhood $U\subset\Lambda M$ of the critical circle $S^1\cdot\gamma$
such that
\begin{equation*}
    \lochom_*(S^1\cdot\gamma)\simeq H_*\left(U,U\cap E^{-1}([0,E(\gamma))\right).
\end{equation*}
If $\gamma$ and all its iterates are homologically invisible,
Theorem~\ref{thm:mountainpass} is thus true for $m_\gamma = 1$.

\begin{proof}[Proof of Theorem~\ref{thm:Bangert}] Assume there are only
    finitely many prime closed geodesics 
    $\gamma_1,\ldots,\gamma_k$ which have homologically visible
    iterates and which intersect
    $M\setminus (C_- \cup C_+)$. We will now derive a contradiction from this
    assumption. We will define a suitable
    finite-dimensional approximation $\Omega=\Omega(\mathcal{O},\kappa,j)$.
    Now as the statement of Theorem~\ref{thm:mountainpass} remains true in a
    finite-dimensional approximation, we get that there exists $m_0\in\N$ such
    that for all integers $m\ge m_0$ and for all $i\in\{1,\ldots,k\}$ the following
    holds:
    \begin{enumerate}[i)]
        \item \label{it:existsU}There exists a neighborhood $U$ of
            $S^1\cdot\gamma_i^m$ in $\Omega$
            such that $U \cap E^{-1}([0,E(\gamma_i^m)))$ is connected.
    \end{enumerate}
    Set $A:=\max\{E(\gamma_i^{m_0})\,|\, i\in\{1,\ldots,k\}\}$,
    and notice that $A$ is larger than the energy of a closed geodesic
    of mountain pass type.
    We fix an identification of $\pi_1(M)$ with $\Z$ and denote by
    $[\gamma]\in\Z$ the class of a loop $\gamma\in\Lambda M$.
    We define the following sets of curves:
    \begin{align*}
        P_j^\pm := \{\gamma \in \Omega\,|\, \gamma(S^1) \subset
        \text{int}(C_\pm),\ [\gamma]=j \}.
    \end{align*}
    In the following for each $U,V\subset M$, we will denote
    \begin{equation*}
        \dist(U,V) := \inf_{x\in U,\;y\in V} d(x,y).
    \end{equation*}
    Choose $\delta>0$. Then
    there exists an $n\in\N$ such that for any curve $\gamma\in P_n^\pm$ and
    $\dist(\gamma(S^1),M\setminus(C_-\cup C_+))<\delta$ it holds that
    $E(\gamma)\ge A$. We can now say how exactly the finite-dimensional
    approximation has to be chosen:
    \begin{itemize}
        \item Choose a $\kappa>0$ large enough such that there exists a
            homotopy $h:[0,1]\to E^{-1}([0,\kappa))$ in $\Omega$ from $h_0 \in
            P_n^-$ to $h_1 \in P_n^+$ with 
            \begin{equation*}
                \dist\left(h_t(S^1),M\setminus (C_- \cup
                C_+)\right) < \delta,
                \quad \forall t\in[0,1].
            \end{equation*}
        \item Set $\mathcal{O}:=\{p\in M \,|\, \dist(p,M\setminus(C_-\cup
            C_+))<R\}$, where $R>2\kappa^\frac{1}{2}+\delta$ such that $\mathcal{O}$ contains
            $\gamma_1,\dots,\gamma_k$.
            \item Choose $k$ such that the $(-\nabla E)$-flow of the
                finite-dimensional approximation respects $C_\pm$, as described
                above.
        \end{itemize}
        A technical issue is given by the fact that the gradient flow of
        $-\nabla E$ may not be defined for all times as the sublevel sets of
        $E|_\Omega$ are not compact. Ultimately we are only going to be
        interested in curves intersecting the compact set $M\setminus(C_-\cup
        C_+)$, i.e. the subset
        \begin{align*}
            K:=\{\gamma\in\Omega \,|\, \gamma(S^1)\cap (M\setminus(C_-\cup
            C_+))\neq\varnothing\}
        \end{align*}
        of $\Omega$. We introduce a smooth function $g:\Omega\to[0,1]$ with the property that
        \begin{equation*}
            \begin{cases}
                g(\gamma) = 1 &\text{if } \dist(\gamma(S^1), K)\le \frac 12
                \kappa^{\frac 12}, \\
                g(\gamma) = 0 &\text{if } \dist(\gamma(S^1), K)> \frac 32 \kappa^{\frac 12}.
            \end{cases}
        \end{equation*}
        Then the flow $\phi_t$ of $-g\nabla E$ is defined for all times
        $t\ge 0$ and coincides with the negative gradient flow for curves in
        $K$. Two crucial observations about the set $K$ are the following:
        firstly, for all $\bar\kappa<\kappa$ the set $K \cap
        E^{-1}([0,\bar\kappa])$ is compact. Secondly, if $\phi_t(\gamma) \in K$
        for some $\gamma\in\Omega$ and some time $t\ge 0$, we already have
        $\gamma\in K$ as the flow $\phi_t$ respects the convex sets $C_\pm$.
        From this it follows:
        \begin{enumerate}[i)]
            \setcounter{enumi}{1}
        \item\label{it:existsT} Let $0<\kappa_0<\kappa_0+\varepsilon<\kappa$.
            Let $V$ denote a neighborhood of the closed geodesics in $K$ of
            energy $\kappa_0$. Suppose there is no closed geodesic in $K \cap
            E^{-1}((\kappa_0,\kappa_0+\varepsilon])$. Then there exists a time
            $\tau>0$, such that
            \begin{align*}
                \phi_\tau\left(E^{-1}\left([0,\kappa_0+\varepsilon]\right)\right)\cap
                K \subset
                E^{-1}([0,\kappa_0)) \cup V.
            \end{align*}
    \end{enumerate}

This is just the deformation lemma; for a proof see for instance
\cite[Lemma~3.4]{Str90}. We are now set to complete the proof of the theorem.
Define the set of homotopies 
\begin{align*} \Pi := \left\{\beta \,|\,
    \beta:[0,1]\to\Omega,\, \beta_0\in P_n^-, \, \beta_1\in P_n^+\right\}.
\end{align*} 
Note that $\Pi$ is not empty, as $h\in\Pi$. Furthermore,
$\phi_t\circ\beta \in \Pi$ for all $\beta\in\Pi$ and all $t\ge 0$ as the flow
respects the convex sets $C_\pm$ and therefore $\phi_t(\beta_0) \in P_n^-$
and $\phi_t(\beta_1) \in P_n^+$ for all $t\ge 0$. Define now 
\begin{align*}
    \kappa_0 := 
    \inf\limits_{\beta\in\Pi}\max\limits_{\substack{t\in[0,1] \\ \beta_t\in K}}
E(\beta_t) \, .  
\end{align*} 

By definition of $\kappa$, one has $\kappa_0<\kappa$.
For every $\beta\in\Pi$ for time
$t_0:=\min\{t\in[0,1]\,|\,\beta_t\notin P_n^-\}$ it holds that $\beta_{t_0}\in
K$ and $E(\beta_{t_0})\ge A$ (as $E$ and $\beta$ are continuous and there
exists a sequence $(t_k) \nearrow t_0$ such that $\beta_{t_k}\in P_n^-$
and $\dist(\beta_{t_k},M\setminus(C_- \cup C_+))<\delta$). Consequently, we get
$\kappa_0\ge A$. 
Since $\kappa_0<\kappa$, for $\varepsilon>0$ small enough,
the subset $K\cap E^{-1}([0,\kappa_0+\varepsilon])$ is compact
and there are only finitely many $S^1$-orbits of closed geodesics inside
(we assumed every orbit to be isolated).
Let $\{S^1\cdot d_j\}_{1\le j\le l}$ denote the critical
circles of energy $\kappa_0$ in $K$.
By definition of $A$ and by using \ref{it:existsU}) (when $d_j$ is homologically visible),
there exist disjoint neighborhoods
$U_j$ of the $S^1\cdot d_j$'s such that $U_j \cap E^{-1}([0,\kappa_0))$ is connected
for all $j$.
Since $\partial C_\pm$ are not totally geodesic, we know that the $d_j$'s are not
contained in $\partial K$ and we can assume that $U_j \subset K$.
Now because
there are only finitely many closed geodesics in
$K\cap E^{-1}([0,\kappa_0+\varepsilon])$
for $\varepsilon>0$ small enough,
one can take $\varepsilon>0$ such that there is no
closed geodesic in $K \cap E^{-1}((\kappa_0,\kappa_0+\varepsilon])$. By the
definition of $\kappa_0$ there exists a homotopy $\beta\in\Pi$ satisfying
$E(\beta_t)\le \kappa_0+\varepsilon$
for all $t\in[0,1]$ such that $\beta_t\in K$.
Choose
neighborhoods $V_j$ of $S^1\cdot d_j$ such that $\overline{V_j}\subset
\text{int}(U_j)$ and use property \ref{it:existsT}) on the neighborhood
$V:=\bigcup_{k=1}^l V_j$ of closed geodesics of energy $\kappa_0$ in $K$ to
obtain a $\tau>0$ with the property that for the homotopy 
$\phi_\tau\circ\beta$ we have that
$(\phi_\tau\circ\beta)_t \in K$ implies $E((\phi_\tau\circ\beta)_t)<\kappa_0$ or
$(\phi_\tau\circ\beta)_t \in V$. Now $(\phi_\tau\circ\beta)^{-1}(V) = \bigcup_{r=1}^m
(t_r,t'_r)$ and by our choice of the $V_j$ we have
$(\phi_\tau\circ\beta)([t_r,t'_r])\subset U_j$ and for the endpoints
$(\phi_\tau\circ\beta)_{t_r},(\phi_\tau\circ\beta)_{t'_r}\in U_j \cap
E^{-1}([0,\kappa_0))$ for some $j\in\{1,\ldots,l\}$ (which is why we applied
\ref{it:existsT}) only to $V$ and not to $\bigcup_{j=1}^l U_j$ directly). Now,
by using \ref{it:existsU}) if $d_j$ is homologically visible,
we know that $U_j \cap E^{-1}([0,\kappa_0))$ is connected
and consequently we can replace $(\phi_\tau\circ\beta)|_{[t_r,t'_r]}$ by a path in
$E^{-1}([0,\kappa_0))$ with the same endpoints.  After $m$ steps we obtain a homotopy
$\hat\beta:[0,1]\to\Omega$ such that
$E(\hat\beta_t)<\kappa_0$ when $\hat\beta_t\in K$.
Since $(\phi_\tau\circ\beta)_0,(\phi_\tau\circ\beta)_1
\notin K$ it follows that $(\phi_\tau\circ\beta)_0,(\phi_\tau\circ\beta)_1 \notin
\bigcup_{j=1}^l U_j$ and therefore $\hat\beta_0 \in P_n^-$ ,$\hat\beta_1 \in
P_n^+$, hence $\hat\beta\in\Pi$. This contradicts the minimality of $\kappa_0$.
\end{proof}

\section{Contractible and intersecting closed geodesics}\label{se:contractible}

Here $M$ still denotes a complete Riemannian cylinder.
We assume that there exists a contractible closed geodesic
$c\in\Lambda M$.
Let us consider the unbounded components of $M\setminus c(S^1)$.
Since $c(S^1)$ is bounded, there are at most two distinct unbounded components.
If there are two distinct unbounded components $C_-$ and $C_+$,
one can assume that $C_-$ is a neighborhood of $-\infty$ and
$C_+$ is a neighborhood of $+\infty$.
By $C_\pm$ we will mean any of these two neighborhoods.
Then $\partial C_\pm$ is a broken geodesic with angles strictly less than $\pi$
inside $C_\pm$
since $c$ is a closed geodesic
(see Figure~\ref{fig:label} for an instance of $\partial C_+$).
Hence $C_\pm$ is locally convex.
Moreover if the boundary were totally geodesic,
then $\partial C_\pm$ would be parametrised by $c$
which is impossible for $c$ is contractible.
We can thus apply Theorem~\ref{thm:Bangert} in this case.

\begin{figure}
\begin{center}
\begin{small}
\def\svgwidth{0.6\textwidth}
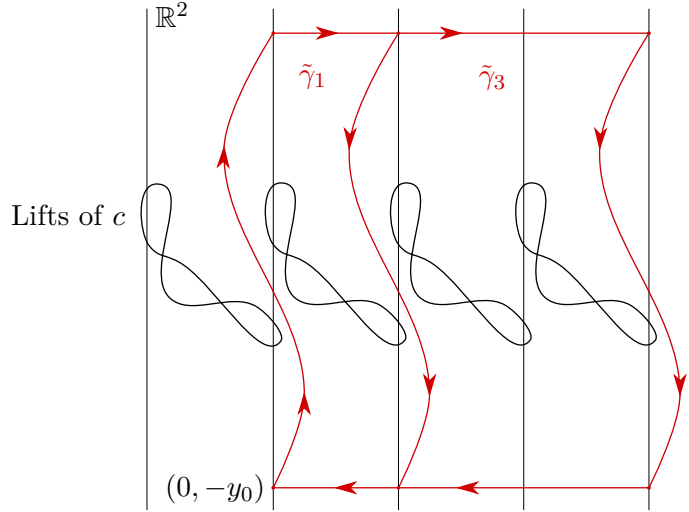
\end{small}
\caption{The family of loops $(\tilde{\gamma}_n)$}
\label{fig:contractible}
\end{center}
\end{figure}
We now assume that $M\setminus c(S^1)$ has only one unbounded component $C$.
Let us identify $M$ with $S^1\times\R$ in the remaining of this proof
in order to fix notations.
Let $\pi : \R^2 \to S^1\times\R$ be the universal cover of $S^1\times\R$.
By compactness of $c(S^1)$, there exists $A>0$ such that $c(S^1)\subset S^1\times
(-A,A)$.
Let $y_0 > A$, since $S^1\times (-\infty,-A)$ and $S^1\times (A,+\infty)$
belong to the same component of $M\setminus c(S^1)$, there exists
a smooth path $\alpha:[0,1]\to M\setminus c(S^1)$ such that
$\alpha(0)=(0,-y_0)$ and $\alpha(1)=(0,y_0)$.
Let $\beta_0$ be the smooth lift of $\alpha$ in $\R^2$
such that $\beta_0(0) = (0,-y_0)$ and $\beta_0(1) = (n_0,y_0)$
for some $n_0 \in \Z$ that we can take equal to $n_0=0$ by 
chaining $\alpha$ with $t\mapsto (tn_0\bmod 1,y_0)$.
Let $\delta_{n,\pm}:[0,1]\to\R^2$ be the path $t\mapsto (nt,\pm y_0)$
and $\beta_n:[0,1]\to\R^2$ be the family of lifts
$\beta_n := (n,0) + \beta_0$, $n\in\N$.
We define the family of loops $\tilde{\gamma}_n\in\Lambda\R^2$
by 
\begin{equation*}
    \tilde{\gamma}_n := \beta_0\cdot \delta_{n,+} \cdot \beta_n^{-1} \cdot
    \delta_{n,-}^{-1}.
\end{equation*}
They project to $\gamma_n := \pi\circ\tilde{\gamma}_n$ in
$M\setminus c(S^1)$.
Let $q_0\in\R^2$ be a lift of some point of $c(S^1)$ and
define $q_n := q_0 + (n,0)$.
Then the first homology group $H_1(\R^2\setminus\{q_n\}_{n\in\Z})$ is
the free abelian group with generators $(g_n)_{n\in\Z}$,
and by construction the class of $\tilde{\gamma}_n$
is $g_1 + g_2 + \cdots + g_n$.
Since the covering transformations of $\R^2\setminus\{q_n\}_{n\in\Z}\to S^1\times\R
\setminus \pi(q_0)$, which form a group isomorphic to $\Z$,
act on the first homology group by $k\cdot g_i = g_{i+k}$,
we see that, if $n\neq m$ and $k,l\in\Z^*$,
the iterated loops $\gamma_n^k$ and $\gamma_m^l$ are 
not freely homotopic in $M\setminus \pi(q_0)$ and
hence in the unbounded component $C$ of $M\setminus c(S^1)$.
We want to apply the $(-\nabla E)$-flow to $\gamma_n$,
$n\in\N$.
Since $c$ is a closed geodesic, the unbounded component $C$ of $M\setminus
c(S^1)$ is a locally convex neighborhood. Hence $\Lambda C$ is preserved
by the $(-\nabla E)$-flow.
Since $M$ is complete, the set of points at distance less than $\ell >0$
from $c(S^1)$ is compact.
Moreover, the image of $\gamma_n$ by the flow (when defined) is kept inside this compact
set for $\ell = \frac{1}{2}\sqrt{E(\gamma_n)}$.
Thus one can apply the $(-\nabla E)$-flow to $\gamma_n$ at all time $t>0$
and ultimately get a closed geodesic of $C$ homotopic to $\gamma_n$.
We thus get a family of closed geodesics which are not
homotopic and not iterations of each other.

Now that Theorem~\ref{thm:main} is proved under hypothesis \ref{it:contractible},
in order to prove it when there is one self-intersecting closed geodesic $c$
or two intersecting ones $c_1$ and $c_2$, one can assume that these geodesics
are not contractible.
Therefore, in both respective cases, $M\setminus c(S^1)$ or $M\setminus(c_1(S^1)\cup c_2(S^1))$
has exactly two unbounded connected components $C_-$ and $C_+$,
which are locally convex by construction.
The intersection hypothesis then implies that none of the boundaries $\partial C_\pm$
is totally geodesic.
Hence the conclusion follows by applying Theorem~\ref{thm:Bangert}.

\section{Geodesic of non-zero average index}\label{se:mind}

We assume that there exists a closed geodesic $c\in\Lambda M$ of average index
$\mind(c) > 0$.  If $c$ is contractible or self-intersecting, we already know
that there are infinitely many closed geodesics.
Let us assume that $c$ is an embedded curve generating $\pi_1(M)\simeq\Z$.
By a slight abuse of notation, we identify the loop $c:S^1\to M$
with its lift $\R\to M$.

\begin{lem}\label{lem:delta}
    There exist $k\in\N^*$ and $\delta \in (0,1]$ such that
    for all $s\in \R$, there exists a
    Jacobi field $J:\R\to c^*TM$ of $c$
    with
    \begin{enumerate}
        \item\label{it:Js} $J(s)=0$,
        \item $J|_{(s,s+\delta)}$ non-vanishing,
        \item\label{it:Js+t} $J(s+t)=0$ for some $t\in [\delta,k]$.
    \end{enumerate}
\end{lem}

\begin{proof}
Given a closed geodesic $\gamma\in\Lambda M$
and $s\in S^1$,
let us denote by $\gamma_s : [0,1]\to M$
the geodesic path $\gamma_s(t) := \gamma(s+t)$.
Since $\mind(c)>0$,
Bott iteration inequality (\ref{eq:Bott}) and the concavity bound
(\ref{eq:conc}) imply
that there exists $k\in\N^*$ such that
\begin{equation*}
    \ind_\Omega((c^k)_s) \geq 1,
    \quad \forall s\in S^1.
\end{equation*}
Let us fix such a $k\geq 1$.
Then according to the Morse index theorem (\ref{eq:MorseIndexThm}) for all $s\in\R$
we can find a non-zero Jacobi field $J$ of $c$
satisfying conditions (\ref{it:Js}) and (\ref{it:Js+t}).
Let $r>0$ be the injectivity radius along the curve $c$.
We define $\delta>0$ by
\begin{equation*}
    \delta = \frac{r}{|\dot{c}(0)|},
\end{equation*}
(we recall that $c$ has constant speed $|\dot{c}(0)|$).
Then by definition of $r$, $\delta < 1$.
Given any Jacobi field $J\in\J(c)\setminus 0$,
there exists a smooth family of geodesics $(\gamma_s)$
with $\gamma_0=c$ such that $J=\partial_s\gamma_s|_{s=0}$.
Suppose that $J(0)=J(t)=0$ with $t<\delta$,
then for some $s$ close to $0$, $\gamma_s$
must intersect $c$ at $c(t_0)$ and $c(t_1)$
with $2|t_0|$ and $2|t_1-t|$ strictly less than $\delta-t$
so that $|t_1-t_0|<\delta$ which contradicts the definition
of the injectivity radius $r$. Hence $\delta$ fulfills the condition
of the lemma for any non-zero Jacobi fields vanishing at $0$.
\end{proof}

In order to fix notation,
let us identify the image of the loop $c$ to $S^1\times \{ 0\}$,
with $c(s)=(s,0)$ for $s\in S^1$,
so that $M\setminus c(S^1)$ is the disjoint union of
the neighborhood $S^1\times (-\infty,0)$ of $-\infty$
and the neighborhood $S^1\times (0,+\infty)$ of $+\infty$
(we only need this identification to be a homeomorphism).
We now use Lemma~\ref{lem:delta} to 
find $n\leq \lceil{1/k(\delta-\varepsilon)}\rceil +1$
geodesic chords $\alpha_1,\ldots,\alpha_n$
lying inside $S^1\times [0,+\infty)$ 
and intersecting $c$ only at endpoints
so that the unbounded component $C_+$ of
$S^1\times (0,+\infty) \setminus \bigcup_i \alpha_i([0,1])$
is locally convex and not a closed geodesic.
We can do the same on the other side and eventually
find two locally convex neighborhoods $C_\pm$ of $\pm\infty$
whose boundaries are not totally geodesic
and with non-intersecting closure.
To make the statement precise,
let $p : S^1 \to [0,1)$ be the natural bijection of $S^1=\R/\Z$
to the fundamental domain $[0,1)$ and
$\pi : M\simeq S^1\times\R \to [0,1)$ be the induced projection
so that $\pi\circ c(t) = t$ for $t\in[0,1)$.

\begin{figure}
\begin{center}
\begin{small}
\def\svgwidth{0.6\textwidth}
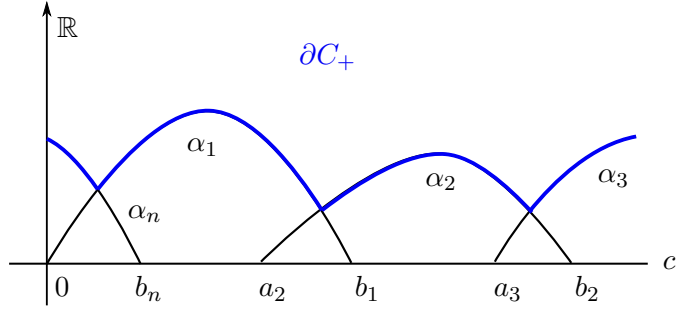
\end{small}
\caption{Construction of the locally convex neighborhood $C_+$}
\label{fig:label}
\end{center}
\end{figure}
\begin{proof}
    Let $k\in\N^*$ and $\delta>0$ be given by Lemma~\ref{lem:delta}.
    Let $\varepsilon < \frac{\delta}{4}$ be a small non-zero real number.
    Up to a translation of the parametrisation of $c$,
    we construct inductively a family
    $\alpha_1,\dotsc,\alpha_n$, $n\leq \lceil{1/k(\delta-\varepsilon)}\rceil$,
    of geodesic chords satisfying:
    \begin{enumerate}
        \item \label{it:initial} $\alpha_1(0)=c(0)$.
        \item\label{it:interc}
            $\alpha_i:[0,1]\to S^1\times [0,+\infty)$ is
            a geodesic path intersecting $c$ only at
            $\alpha_i(0)$ and $\alpha_i(1)$.
        \item\label{it:delta}
            If $a_i:=\pi\circ\alpha_i(0)$ and $b_i:=\pi\circ\alpha_i(1)$
            denote the positions of endpoints, then
            \begin{equation*}
                b_i-a_i>\delta-2\varepsilon \quad\text{ and }\quad
                [a_i,b_i]\subset \pi\circ\alpha_i([0,1])
            \end{equation*}
            for $1\leq i\leq n$.
            And for $i=n$, 
            \begin{equation*}
                [0,b_n]\cup [a_n,1) \subset
                \pi\circ\alpha_n([0,1]). 
            \end{equation*}
        \item\label{it:order}
            The position of endpoints satisfy
            \begin{equation*}
                \begin{cases}
                    a_i < a_{i+1} < b_i <b_{i+1},& \quad 1\leq i\leq n-2,\\
                    a_{n-1} < a_n < b_{n-1},& \\
                    a_1 = 0 \leq b_n.
                \end{cases}
            \end{equation*}
    \end{enumerate}
    Suppose $\alpha_1,\dotsc,\alpha_q$ satisfies properties (\ref{it:interc})
    and (\ref{it:delta}) and relation of property (\ref{it:order}) for $i\in\{
    1,\dotsc, q-1\}$.  To construct $\alpha_1$, we take $b_0 := 0$ in the
    following then we translate the parametrisation of $c$  once for all so
    that property (\ref{it:initial}) is fulfilled.  By definition, the position
    $b_q\in[0,1)$ satisfies $c(b_q)=\alpha_q(1)$.  According to
    Lemma~\ref{lem:delta}, there exists a Jacobi field along $c$ such that
    $J(b_q-\varepsilon) = 0$, $J$ does not vanish on
    $I:=(b_q-\varepsilon,b_q+\delta -\varepsilon)$ and $J(b_q-\varepsilon+t)=0$
    for some $t\in[\delta,1]$.  Up to a change of sign, one can assume that
    $J|_I$ is pointing inside $S^1\times (0,+\infty)$.  Since there exists a
    smooth family $(\beta_s)_{s\in (-1,1)}$ of geodesic paths such that $J|_I =
    \frac{\partial \beta_s}{\partial s}|_{s=0}$, it implies that there exists a
    geodesic path $\alpha_{q+1} : [0,1]\to S^1\times [0,+\infty)$ intersecting
    $c$ (transversally) only at its endpoints with $a_{q+1} < b_q < b_{q+1} -
    \delta - 2\varepsilon$ unless $\pi\circ\alpha_{q+1}([0,1]) =
    [0,b_{q+1}]\cup [a_{q+1},1)$, in this case we set $n:=q+1$.  In the
    exceptional case where $b_n=0$, we construct a last geodesic chord
    $\alpha_{n+1}$ satisfying property (\ref{it:interc}), with at least one
    endpoints different from $c(0)$ and whose image projects on a neighborhood
    of $c(0)\in S^1$.

    Thus we get a family of geodesic chords such that consecutive chords intersect
    in $S^1\times (0,+\infty)$ and the union of their images $\bigcup_i \alpha_i([0,1])$
    projects onto $S^1$. By construction, the unbounded component $C_+$ of
    $S^1\times (0,+\infty)\setminus \bigcup_i\alpha_i([0,1])$ has a boundary
    which is broken geodesic with angles strictly less than $\pi$.
    By symmetry, we get two disjoint neighborhoods of $+\infty$
    and $-\infty$ respectively which are locally convex and
    whose boundaries are not totally geodesic,
    we can thus apply Theorem~\ref{thm:Bangert}.
\end{proof}

\section{Two homologically visible geodesics}\label{se:homological}

Here $M$ denotes a complete Riemannian cylinder.
We fix an identification of $\pi_1(M)$ with $\Z$ and
denote by $[\gamma]\in\Z$ the class of a loop
$\gamma\in\Lambda M$.
We assume that there exists two geometrically distinct 
and homologically visible closed geodesics.
We suppose by contradiction that there are only finitely many (geometrically distinct)
closed geodesics in $M$.
By the previous cases of Theorem~\ref{thm:main},
every prime closed geodesic of $M$ must be embedded,
non-contractible, without intersections with another
closed geodesic, and of zero average index.
Thus the closed geodesics of $M\simeq S^1\times \R$ are naturally ordered
by their intersection with $*\times\R$ where $*$
denotes any point of $S^1$. The order is independent
of the choice of $*\in S^1$.
We will say that two closed geodesics are consecutive
if they are so with respect to this order.
Since we assume that the $S^1$-orbits of closed geodesics are isolated,
given a closed geodesic, one can talk about the next and the previous one
with respect to this order.

\begin{lem}\label{lem:compactcylinder}
    There exists two closed embedded geodesics $c_1$ and $c_2$ of $M$
    with degree $[c_1]=[c_2]=1$ bounding a compact locally convex cylinder
    $C\simeq S^1\times [0,1]$ such that
    \begin{enumerate}
        \item $c_1$ is a local minimum of $E|_{\Lambda C}$,
        \item $c_2$ is not a local minimum of $E|_{\Lambda C}$,
        \item $c_1$ and $c_2$ are the only
            closed geodesics of $M$ inside $C$
            that have homologically visible iterates.
    \end{enumerate}
\end{lem}

\begin{proof}
    We first show that two consecutive closed
    geodesics among closed geodesics that possess homologically visible
    iterates cannot be both local minima of $E|_{\Lambda C'}$
    if $C'$ is the compact cylinder that they bound.
    By contradiction, let us assume so and let us call
    $\gamma_0$ and $\gamma_1$ these two geodesics.
    Up to a change of parametrization, one can assume that
    $[\gamma_0]=[\gamma_1]$ and thus that these two
    geodesics are homotopic in $\Lambda C'$.
    Let
    \begin{equation*}
        \Pi := \{ h:[0,1]\to\Lambda C' \text{ continuous } |\
            h(0) := \gamma_0 \text{ and }
        h(1) := \gamma_1 \}
    \end{equation*}
    denote the set of homotopy of loops in $C'$ starting at $\gamma_0$
    and ending at $\gamma_1$.
    We consider the following min-max:
    \begin{equation*}
        \tau = \inf_{h\in\Pi} \max E\circ h.
    \end{equation*}
    Let $e:=\max(E(\gamma_0),E(\gamma_1))$. Since
    $\gamma_0$ and $\gamma_1$ are local minima of $E|_{\Lambda C'}$,
    $\tau > e$.
    By compactness of $C'$, $E|_{\Lambda C'}$ satisfies Palais-Smale
    (alternatively, one can work in the compact finite-dimensional
    manifold of $k$-broken-geodesics of energy $\leq c+\varepsilon$
    for a large $k\in\N$ and $\varepsilon>0$).
    By local convexity of $C'$,
    the $(-\nabla E)$-flow preserves $\Lambda C'$.
    By the minimax principle, $\tau$ is thus a critical value of
    $E|_{\Lambda C'}$ and there exists a homologically visible
    closed geodesic $\gamma\in\Lambda C'$ of energy $\tau$.
    Hence $\gamma_0$ and $\gamma_1$ are not consecutive,
    a contradiction.

    By a similar argument, we show that one out of two consecutive
    closed geodesics among those that possess homologically visible
    iterates
    is a local minimum of
    $E|_{\Lambda C'}$.
    Indeed, otherwise one has that
    \begin{equation*}
        \inf_{\substack{\gamma\in\Lambda C'\\ [\gamma]=1}} E(\gamma)
        < \min(E(\gamma_0),E(\gamma_1)),
    \end{equation*}
    and this infimum is reached for some closed geodesic in $C'$
    by compactness and local convexity of $C'$
    (and this is not a point since $E(\gamma)\geq (2r)^2$
    for all $\gamma\in\Lambda C'$ of degree $[\gamma]=1$
    where $r>0$ denotes the injectivity radius of the compact
    Riemannian manifold with boundary $C'$).
    This new closed geodesic is a local minimum of
    $E$ by definition and thus homologically visible.

    The requirements of the lemma are thus fulfilled by taking
    any consecutive homologically visible closed geodesics.
\end{proof}

\begin{proof}[Proof of Theorem~\ref{thm:main}]
    Let $c_1$ and $c_2$ be closed geodesics of $M$
    satisfying Lemma~\ref{lem:compactcylinder}.
    We will reach a contradiction by finding a
    homologically visible geodesic which is not
    $c_1$ or $c_2$ and arbitrarily close to $C$.

    Let $x\in\interior{C}$ and let $\gamma_1\in\Lambda C$
    be the loop of degree $[\gamma_1]=1$ based at $x$
    of minimal length.
    It exists by local convexity and compactness of $C$.
    The loop $\gamma_1$ is not a periodic geodesic
    (this is a geodesic as a path $[0,1]\to C$
    but not as a loop $S^1\to C$)
    since there is no local minimum of $E|_{\Lambda C}$
    but $c_1$.
    This loop lies inside $\interior{C}$ so that
    either the connected component of $C\setminus \gamma_1(S^1)$ 
    containing $c_1$ or the connected component containing $c_2$
    is locally convex -- depending on the angle of $\gamma_1$ at
    $\gamma_1(0)=\gamma_1(1)=x$.
    If the connected component containing $c_2$ were locally convex,
    then the infimum of $E$ among loops of degree one
    lying inside the locally convex compact cylinder bounded
    by $\gamma_1$ and $c_2$
    would give a closed geodesic loop $\neq c_1$
    which would be a local minimum.
    Thus the connected component of $C\setminus\gamma_1(S^1)$
    containing $c_1$ is a locally convex compact cylinder.
    Hence the unbounded component of $M\setminus\gamma_1(S^1)$
    containing $c_1$ is a locally convex neighborhood of $-\infty$
    which is not totally geodesic
    since $\gamma_1$ is not a closed geodesic.

    \begin{figure}
        \begin{center}
            \begin{small}
                \def\svgwidth{0.6\textwidth}
                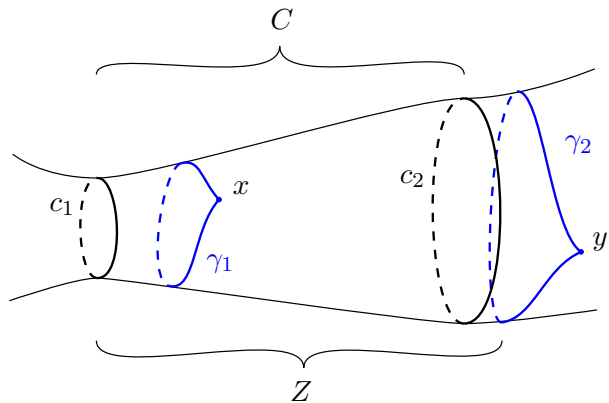
            \end{small}
            \caption{Construction of cylinder $Z$}
            \label{fig:cylinderZ}
        \end{center}
    \end{figure}
    Let $c_3$ be the closed geodesic succeeding $c_2$
    among closed geodesic (possibly non homologically visible),
    if it exists.
    Let $C'$ be either the compact cylinder that $c_2$ and $c_3$ bound
    or the infinite cylinder $\simeq S^1\times [0,+\infty)$
    with boundary $c_2$ and ending at $+\infty$, depending on
    the existence of $c_3$
    (so that $C\cap C' = c_2(S^1)$ in both cases).
    Let $y\in\interior{C'}$ and let $\gamma_2\in\Lambda C'$
    be a loop of degree $[\gamma_2]=1$ based at $y$ of
    minimal length.
    Since $C'$ is complete and locally convex, it exists.
    It cannot be a closed geodesic for $c_3$ succeeds $c_2$.
    One of the two unbounded components of $M\setminus\gamma_2(S^1)$
    is thus locally convex, depending on the angle of $\gamma_2$
    at $\gamma_2(0)=\gamma_2(1)=y$.
    If the neighborhood of $+\infty$ was the locally convex one,
    by Theorem~\ref{thm:Bangert} applied to the locally convex
    neighborhood of $-\infty$ defined above with $\gamma_1$
    and this neighborhood of $+\infty$,
    there would be infinitely many closed geodesics.
    Thus the neighborhood of $-\infty$ is the locally convex
    unbounded component of $M\setminus \gamma_2(S^1)$.
    Restricting this neighborhood to the compact cylinder
    $C\cup C'$, one gets a compact locally convex cylinder
    $Z$ intersecting only two
    geodesics $c_1$ and $c_2$ that possess homologically visible
    iterates, moreover
    $c_1(S^1) \subset \partial Z$ and
    $c_2(S^1) \subset \interior{Z}$.

    Let $k\in\Z^*$ be such that $\lochom_*(S^1\cdot c_2^k)\neq 0$.
    Let $\Lambda_m\subset \Lambda Z$ be the connected component 
    of loops $\gamma\in\Lambda Z$ of degree $[\gamma]=m$.
    Let $\psi_m : \Lambda_k\to\Lambda_{km}$ be the iteration map
    $\psi_m(\gamma):=\gamma^{m}$.
    According to Bangert-Klingenberg theorem (\ref{eq:BK}),
    there exists $m_0\in\N$ above which for all
    $m\geq m_0$ there exists $e_m > m^2 E(c_2^k)$ such that
    the composition
    \begin{equation*}
        \lochom_*(S^1\cdot c_2^k) \xrightarrow{(\psi_m)_*}
        \lochom_*(S^1\cdot c_2^{km}) \xrightarrow{\textup{inc}_*}
        H_*\left(\big\{ E|_{\Lambda_{km}} < e_m\big\},
        \left\{ E|_{\Lambda_{km}}< m^2 E(c_2^k) \right\} \right)
    \end{equation*}
    is zero.
    According to Gromoll-Meyer theorem (\ref{eq:GM}),
    since $\mind(c_2^k)=k\mind(c_2)=0$,
    there exists infinitely many $m$ such that
    \begin{equation*}
        (\psi_m)_* : \lochom_*(S^1\cdot c_2^k) \to\lochom_* (S^1\cdot c_2^{km})
    \end{equation*}
    is an isomorphism. Let $m\geq m_0$ be such an integer,
    then the inclusion induces a zero map
    \begin{equation*}
        \lochom_*(S^1\cdot c_2^{km}) \xrightarrow{\textup{inc}_*}
        H_*\left(\big\{ E|_{\Lambda_{km}} < e_m\big\},
        \left\{ E|_{\Lambda_{km}}< m^2 E(c_2^k) \right\} \right),
    \end{equation*}
    which contradicts the fact that $c_2^{km}$ is the
    homologically visible critical points of $E|_{\Lambda_{km}}$ of
    maximal value.
    Indeed, since $Z$ is locally convex, critical points
    of $E|_{\Lambda_{km}}$ are closed geodesics of $Z$ of degree $km$.
    Thus $S^1\cdot c_1^{km}$ and $S^1\cdot c_2^{km}$ are the only homologically
    visible critical circle of $E|_{\Lambda_{km}}$
    (and $E(c_2^{km}) > E(c_1^{km})$ since $c_1$ is the only
    local minimum in $C$).
    Moreover $Z$ is compact and has only isolated closed geodesics,
    we can thus apply Morse theoretical
    arguments since $E|_{\Lambda_{km}}$ has isolated critical circles and
    satisfies Palais-Smale
    or, alternatively, one can restrict $E$ to the finite-dimensional
    subspace of $k$-broken-geodesics of $\Lambda_{km}$ of energy less than $e_m+\varepsilon$
    for some large $k\in\N$ and some $\varepsilon>0$.
    Thus, if $S^1\cdot c_2^{km}$ were the only homologically visible 
    critical circle of energy $\geq m^2E(c_2^k)$,
    Morse deformation lemma would imply $\textup{inc}_*$ to be an isomorphism.
\end{proof}

\section{The case of the plane}\label{se:plane}

Let $M\simeq\R^2$ be a complete Riemannian plane with isolated closed
geodesics. Using what we have seen in the previous sections, we now
give the proof of Theorem~\ref{thm:plane}.

\begin{proof}[Proof of Theorem~\ref{thm:plane}]
When hypothesis \ref{it:plane:selfintersecting}, \ref{it:plane:intersecting}
or \ref{it:plane:mind} is assumed, the conclusion follows from
the same argument as in the case of the cylinder: by construction of
an open neighborhood $C\neq M$ of infinity.
More precisely, this neighborhood $C$ is the unbounded component of
$M\setminus c(S^1)$
or $M\setminus (c_1(S^1)\cup c_2(S^1))$ if $c$ is self-intersecting
or $c_1$ and $c_2$ are intersecting closed geodesics.
In the case when there exists a closed geodesic $c$ of non-zero
average index, $C$ is constructed by ``integrating Jacobi fields''
along $c$ as was done in Section~\ref{se:mind}.

Now, let us assume that all the closed geodesics of $M$ are without
self-intersection and with zero average index. In order to
complete the proof, we must prove that this last case implies the existence of
infinitely many closed geodesics.  Let $c$ be a homologically visible
closed geodesic such that there
is not any homologically visible closed geodesic inside the disk $D$ bounded by
$c$.  Let $G=\bigcup_\gamma
\gamma(S^1)\subset M$ be the union of the images of the closed geodesics
$\gamma$ of $M$.  Let $U$ be the connected component of $M\setminus(D\cup G)$
that contains $c(S^1)$ in its boundary.  Since $U$ contains loops that are not
contractible in $\R^2\setminus D$ (by taking loops close to the boundary
$c(S^1)$), $U$ is not simply connected.  Let $y\in M$ and let $\gamma\in\Lambda
\overline{U}$ be a loop minimizing the length among the non-contractible loops
of $\overline{U}$ based at $y$ (it exists since $\overline{U}$ is complete).
Since $\partial U$ is a disjoint union of closed geodesics, $\gamma$ lies in the
interior of $U$ and is a geodesic path.  Depending on the angle that $\gamma$ makes
at $y$, either the unbounded component of $M\setminus \gamma(S^1)$ is
locally convex and not totally geodesic or the bounded component containing $c$
is locally convex.  In the first case, one can apply Bangert's theorem to
complete the proof. 

We can thus assume that $c$ lies in the interior of a
compact and locally convex subset $K\subset M$ and that $c$ is the only
homologically visible closed geodesic of $K$. 
Since $\mind(c)=0$, the local homology groups $C_d(S^1\cdot c^m)$ are
trivial in degrees $d\geq 4$ for all $m\in\N$.
Let $d\in\{0,1,2,3\}$ be the maximal degree such that $C_d(S^1\cdot c^m)\neq 0$
for some $m\in\N^*$. Let $k\in\N^*$ be such that $C_d(S^1\cdot c^k)\neq 0$.
According to Gromoll-Meyer theory, there exists infinitely many
$m\in\N^*$ such that the map induced by the iteration map
\begin{equation*}
    (\psi_m)_* : \lochom_*(S^1\cdot c^k) \to\lochom_* (S^1\cdot c^{km})
\end{equation*}
is an isomorphism.
As above, according to Bangert-Klingenberg theorem (\ref{eq:BK}),
there exists $m_0\in\N^*$ such that, for all such $m\in\N^*$ greater than $m_0$,
the inclusion of sublevel sets of $E|_{\Lambda K}$ induces the zero map
\begin{equation*}
    \lochom_*(S^1\cdot c^{km}) \xrightarrow{\textup{inc}_*}
    H_*\left(\big\{ E|_{\Lambda K} < e_m\big\},
    \left\{ E|_{\Lambda K}< m^2 E(c^k) \right\} \right),
\end{equation*}
for some $e_m > m^2 E(c^k)$.
Thus, for such an $m$, the long exact sequence of the triple
\begin{equation*}
    \left( \big\{ E|_{\Lambda K} < e_m\big\},
        \left\{ E|_{\Lambda K}< m^2 E(c^k) \right\}\cup S^1\cdot c^{km},
    \left\{ E|_{\Lambda K}< m^2 E(c^k) \right\} \right)
\end{equation*}
implies that
\begin{equation*}
    H_{d+1} \left( \big\{ E|_{\Lambda K} < e_m\big\},
    \left\{ E|_{\Lambda K}< m^2 E(c^k) \right\}\cup S^1\cdot c^{km} \right) \neq 0.
\end{equation*}
Therefore, by Morse deformation lemma applied to the smooth map $E|_{\Lambda K}$
which satisfies the Palais-Smale condition
and whose anti-gradient flow preserves $\Lambda K$
(by compactness and local convexity of $K$), there must be a closed geodesic
$\gamma\in\Lambda K$ such that $\lochom_{d+1}(S^1\cdot \gamma)\neq 0$.
By maximality of $d$, $\gamma$ and $c$ are geometrically distinct.
But $c$ is the only homologically visible closed geodesic of $K$,
a contradiction.
\end{proof}

\bibliographystyle{amsplain}
\bibliography{../biblio}
%\bibliography{biblio} 

\end{document}

%% file: contractible.pdf_tex
%% Creator: Inkscape inkscape 0.92.2, www.inkscape.org
%% Accompanies image file 'contractible.pdf' (pdf, eps, ps)
\begingroup%
  \makeatletter%
  \providecommand\color[2][]{%
    \errmessage{(Inkscape) Color is used for the text in Inkscape, but the package 'color.sty' is not loaded}%
    \renewcommand\color[2][]{}%
  }%
  \providecommand\transparent[1]{%
    \errmessage{(Inkscape) Transparency is used (non-zero) for the text in Inkscape, but the package 'transparent.sty' is not loaded}%
    \renewcommand\transparent[1]{}%
  }%
  \providecommand\rotatebox[2]{#2}%
  \ifx\svgwidth\undefined%
    \setlength{\unitlength}{386.54107866bp}%
    \ifx\svgscale\undefined%
      \relax%
    \else%
      \setlength{\unitlength}{\unitlength * \real{\svgscale}}%
    \fi%
  \else%
    \setlength{\unitlength}{\svgwidth}%
  \fi%
  \global\let\svgwidth\undefined%
  \global\let\svgscale\undefined%
  \makeatother%
  \begin{picture}(1,0.7385602)%
    \put(0,0){\includegraphics[width=\unitlength,page=1]{contractible.pdf}}%
    \put(-0.00304685,0.41935196){\color[rgb]{0,0,0}\makebox(0,0)[lb]{\smash{Lifts of $c$}}}%
    \put(0.20819387,0.71195955){\color[rgb]{0,0,0}\makebox(0,0)[lb]{\smash{$\R^2$}}}%
    \put(0.41786988,0.62902803){\color[rgb]{0.81568627,0,0}\makebox(0,0)[lb]{\smash{$\tilde{\gamma}_1$}}}%
    \put(0.68074729,0.62746319){\color[rgb]{0.81568627,0,0}\makebox(0,0)[lb]{\smash{$\tilde{\gamma}_3$}}}%
    \put(0.22227663,0.02034169){\color[rgb]{0,0,0}\makebox(0,0)[lb]{\smash{$(0,-y_0)$}}}%
    \put(0.43508207,0.04850715){\color[rgb]{0,0,0}\makebox(0,0)[lb]{\smash{}}}%
  \end{picture}%
\endgroup%

%% file: Jacobi.pdf_tex
%% Creator: Inkscape inkscape 0.92.2, www.inkscape.org
%% Accompanies image file 'Jacobi.pdf' (pdf, eps, ps)
\begingroup%
  \makeatletter%
  \providecommand\color[2][]{%
    \errmessage{(Inkscape) Color is used for the text in Inkscape, but the package 'color.sty' is not loaded}%
    \renewcommand\color[2][]{}%
  }%
  \providecommand\transparent[1]{%
    \errmessage{(Inkscape) Transparency is used (non-zero) for the text in Inkscape, but the package 'transparent.sty' is not loaded}%
    \renewcommand\transparent[1]{}%
  }%
  \providecommand\rotatebox[2]{#2}%
  \ifx\svgwidth\undefined%
    \setlength{\unitlength}{235.7145719bp}%
    \ifx\svgscale\undefined%
      \relax%
    \else%
      \setlength{\unitlength}{\unitlength * \real{\svgscale}}%
    \fi%
  \else%
    \setlength{\unitlength}{\svgwidth}%
  \fi%
  \global\let\svgwidth\undefined%
  \global\let\svgscale\undefined%
  \makeatother%
  \begin{picture}(1,0.44875539)%
    \put(0,0){\includegraphics[width=\unitlength,page=1]{Jacobi.pdf}}%
    \put(0.06698843,0.01353982){\color[rgb]{0,0,0}\makebox(0,0)[lb]{\smash{$0$}}}%
    \put(0.18376779,0.01353982){\color[rgb]{0,0,0}\makebox(0,0)[lb]{\smash{$b_n$}}}%
    \put(0.3648604,0.01353963){\color[rgb]{0,0,0}\makebox(0,0)[lb]{\smash{$a_2$}}}%
    \put(0.50533411,0.01353963){\color[rgb]{0,0,0}\makebox(0,0)[lb]{\smash{$b_1$}}}%
    \put(0.7084286,0.01353963){\color[rgb]{0,0,0}\makebox(0,0)[lb]{\smash{$a_3$}}}%
    \put(0.82690043,0.01353982){\color[rgb]{0,0,0}\makebox(0,0)[lb]{\smash{$b_2$}}}%
    \put(0.42409627,0.35541543){\color[rgb]{0,0,0.96078431}\makebox(0,0)[lb]{\smash{$\partial C_+$}}}%
    \put(0.25992823,0.22678888){\color[rgb]{0,0,0}\makebox(0,0)[lb]{\smash{$\alpha_1$}}}%
    \put(0.60857387,0.17601522){\color[rgb]{0,0,0}\makebox(0,0)[lb]{\smash{$\alpha_2$}}}%
    \put(0.86074947,0.18109281){\color[rgb]{0,0,0}\makebox(0,0)[lb]{\smash{$\alpha_3$}}}%
    \put(0.17361303,0.12862656){\color[rgb]{0,0,0}\makebox(0,0)[lb]{\smash{$\alpha_n$}}}%
    \put(0.07037336,0.39603446){\color[rgb]{0,0,0}\makebox(0,0)[lb]{\smash{$\R$}}}%
    \put(0.95569075,0.05295749){\color[rgb]{0,0,0}\makebox(0,0)[lb]{\smash{$c$}}}%
  \end{picture}%
\endgroup%

%% file: cylinderZ.pdf_tex
%% Creator: Inkscape inkscape 0.92.2, www.inkscape.org
%% Accompanies image file 'cylinderZ.pdf' (pdf, eps, ps)
\begingroup%
  \makeatletter%
  \providecommand\color[2][]{%
    \errmessage{(Inkscape) Color is used for the text in Inkscape, but the package 'color.sty' is not loaded}%
    \renewcommand\color[2][]{}%
  }%
  \providecommand\transparent[1]{%
    \errmessage{(Inkscape) Transparency is used (non-zero) for the text in Inkscape, but the package 'transparent.sty' is not loaded}%
    \renewcommand\transparent[1]{}%
  }%
  \providecommand\rotatebox[2]{#2}%
  \ifx\svgwidth\undefined%
    \setlength{\unitlength}{414.63090489bp}%
    \ifx\svgscale\undefined%
      \relax%
    \else%
      \setlength{\unitlength}{\unitlength * \real{\svgscale}}%
    \fi%
  \else%
    \setlength{\unitlength}{\svgwidth}%
  \fi%
  \global\let\svgwidth\undefined%
  \global\let\svgscale\undefined%
  \makeatother%
  \begin{picture}(1,0.57391824)%
    \put(0,0){\includegraphics[width=\unitlength,page=1]{cylinderZ.pdf}}%
    \put(0.05891931,0.28395354){\color[rgb]{0,0,0}\makebox(0,0)[lb]{\smash{$c_1$}}}%
    \put(0.57142329,0.32582481){\color[rgb]{0,0,0}\makebox(0,0)[lb]{\smash{$c_2$}}}%
    \put(0.38216532,0.55192956){\color[rgb]{0,0,0}\makebox(0,0)[lb]{\smash{$C$}}}%
    \put(0.41063774,0.00425359){\color[rgb]{0,0,0}\makebox(0,0)[lb]{\smash{$Z$}}}%
    \put(0.32689519,0.30907643){\color[rgb]{0,0,0}\makebox(0,0)[lb]{\smash{$x$}}}%
    \put(0.85279795,0.23370808){\color[rgb]{0,0,0}\makebox(0,0)[lb]{\smash{$y$}}}%
    \put(0.28837371,0.20188592){\color[rgb]{0,0,0.96078431}\makebox(0,0)[lb]{\smash{$\gamma_1$}}}%
    \put(0.8126016,0.37607028){\color[rgb]{0,0,0.96078431}\makebox(0,0)[lb]{\smash{$\gamma_2$}}}%
  \end{picture}%
\endgroup%

%% file: cylinder.bbl
\providecommand{\bysame}{\leavevmode\hbox to3em{\hrulefill}\thinspace}
\providecommand{\MR}{\relax\ifhmode\unskip\space\fi MR }
\providecommand{\MRhref}[2]{%
  \href{http://www.ams.org/mathscinet-getitem?mr=#1}{#2}
}
\providecommand{\href}[2]{#2}
\begin{thebibliography}{10}

\bibitem{AM17}
Luca Asselle and Marco Mazzucchelli, \emph{On the existence of infinitely many
  closed geodesics on non-compact manifolds}, Proc. Amer. Math. Soc.
  \textbf{145} (2017), no.~6, 2689--2697.

\bibitem{BTZ}
W.~Ballmann, G.~Thorbergsson, and W.~Ziller, \emph{Closed geodesics on
  positively curved manifolds}, Ann. of Math. (2) \textbf{116} (1982), no.~2,
  213--247.

\bibitem{BK83}
V.~Bangert and W.~Klingenberg, \emph{Homology generated by iterated closed
  geodesics}, Topology \textbf{22} (1983), no.~4, 379--388.

\bibitem{Ban80}
Victor Bangert, \emph{Closed geodesics on complete surfaces}, Math. Ann.
  \textbf{251} (1980), no.~1, 83--96.

\bibitem{Ban93}
\bysame, \emph{On the existence of closed geodesics on two-spheres}, Internat.
  J. Math. \textbf{4} (1993), no.~1, 1--10.

\bibitem{BG91}
Vieri Benci and Fabio Giannoni, \emph{Closed geodesics on noncompact
  {R}iemannian manifolds}, C. R. Acad. Sci. Paris S\'{e}r. I Math. \textbf{312}
  (1991), no.~11, 857--861.

\bibitem{BG92}
\bysame, \emph{On the existence of closed geodesics on noncompact {R}iemannian
  manifolds}, Duke Math. J. \textbf{68} (1992), no.~2, 195--215.

\bibitem{Bes78}
Arthur~L. Besse, \emph{Manifolds all of whose geodesics are closed}, Ergebnisse
  der Mathematik und ihrer Grenzgebiete [Results in Mathematics and Related
  Areas], vol.~93, Springer-Verlag, Berlin-New York, 1978, With appendices by
  D. B. A. Epstein, J.-P. Bourguignon, L. B\'{e}rard-Bergery, M. Berger and J.
  L. Kazdan.

\bibitem{Bir66}
George~D. Birkhoff, \emph{Dynamical systems}, With an addendum by Jurgen Moser.
  American Mathematical Society Colloquium Publications, Vol. IX, American
  Mathematical Society, Providence, R.I., 1966.

\bibitem{Bot56}
Raoul Bott, \emph{On the iteration of closed geodesics and the {S}turm
  intersection theory}, Comm. Pure Appl. Math. \textbf{9} (1956), 171--206.

\bibitem{BM19}
Keith Burns and Vladimir~S. Matveev, \emph{Open problems and questions about
  geodesics}, Ergodic Theory and Dynamical Systems (2019), 1–44.

\bibitem{CJ13}
Erasmo Caponio and Miguel~\'{A}ngel Javaloyes, \emph{A remark on the {M}orse
  {T}heorem about infinitely many geodesics between two points}, International
  Meeting on Differential Geometry (2013), 39--48.

\bibitem{Fet52}
A.~I. Fet, \emph{Variational problems on closed manifolds}, Mat. Sbornik N.S.
  \textbf{30(72)} (1952), 271--316.

\bibitem{Fra92}
John Franks, \emph{Geodesics on {$S^2$} and periodic points of annulus
  homeomorphisms}, Invent. Math. \textbf{108} (1992), no.~2, 403--418.

\bibitem{GM68}
D.~Gromoll, W.~Klingenberg, and W.~Meyer, \emph{Riemannsche {G}eometrie im
  {G}rossen}, Lecture Notes in Mathematics, No. 55, Springer-Verlag, Berlin-New
  York, 1968.

\bibitem{GM69a}
Detlef Gromoll and Wolfgang Meyer, \emph{Periodic geodesics on compact
  riemannian manifolds}, J. Differential Geometry \textbf{3} (1969), 493--510.

\bibitem{Had98}
J.~{Hadamard}, \emph{{Les surfaces \`a courbures oppos\'ees et leurs lignes
  g\'eod\'esiques.}}, {Journ. de Math. (5)} \textbf{4} (1898), 27--73 (French).

\bibitem{Hin93}
Nancy Hingston, \emph{On the growth of the number of closed geodesics on the
  two-sphere}, Internat. Math. Res. Notices (1993), no.~9, 253--262.

\bibitem{LF51}
L.~A. Lyusternik and A.~I. Fet, \emph{Variational problems on closed
  manifolds}, Doklady Akad. Nauk SSSR (N.S.) \textbf{81} (1951), 17--18.

\bibitem{Poi05}
Henri Poincar\'{e}, \emph{Sur les lignes g\'{e}od\'{e}siques des surfaces
  convexes}, Trans. Amer. Math. Soc. \textbf{6} (1905), no.~3, 237--274.

\bibitem{Str90}
Michael Struwe, \emph{Variational methods}, Springer-Verlag, Berlin, 1990,
  Applications to nonlinear partial differential equations and Hamiltonian
  systems. \MR{1078018}

\bibitem{Tho78}
Gudlaugur Thorbergsson, \emph{Closed geodesics on non-compact {R}iemannian
  manifolds}, Math. Z. \textbf{159} (1978), no.~3, 249--258.

\end{thebibliography}
